\documentclass[11pt]{amsart}
\usepackage{graphicx}
\usepackage{mathpazo}
\usepackage{amssymb}
\usepackage{epstopdf}
\usepackage[cmtip,all]{xy}
\usepackage{hyperref}
\usepackage{mathrsfs}
\usepackage{enumerate}
\hypersetup{linktocpage}
\newtheorem{prop}{Proposition}[section]
\newtheorem{lemma}[prop]{Lemma}

\usepackage{mathrsfs}
\usepackage{amsthm}
\usepackage{amssymb}
\usepackage{amsfonts}
\usepackage{tikz} 
\usepackage{mathtools}
\usepackage{setspace}
\usepackage{cancel}
\usepackage{chemarrow}
\usetikzlibrary{matrix,arrows} 
\usepackage{amscd}

\newtheorem{theorem}{Theorem}

\newtheorem{cor}[prop]{Corollary}

\theoremstyle{definition}
\newtheorem{example}[prop]{Example}

\newtheorem{defn}[prop]{Definition}

\newcommand{\CC}{\mathbb{C}}
\renewcommand{\P}{\mathbb{P}}

\newcommand{\I}{\mathcal{I}}

\renewcommand{\O}{\mathcal{O}}
\newcommand{\Z}{\mathbb{Z}}
\newcommand{\Q}{\mathbb{Q}}

\newcommand{\E}{\mathcal{E}}
\newcommand{\bI}{\mathbb{I}}

\newcommand{\hilb}{\operatorname{Hilb}}

\newcommand{\Home}{\mathcal{H}om}

\newcommand{\ch}{\operatorname{ch}}

\renewcommand{\H}{\mathcal{H}}

\title{Counting curves on surfaces in Calabi-Yau 3-folds}
\author{Amin Gholampour, Artan Sheshmani and Richard Thomas}
\date{\today}                                   
\begin{document}
\maketitle
\begin{abstract}
Motivated by S-duality modularity conjectures in string theory, we define new invariants counting a restricted class of 2-dimensional torsion sheaves, enumerating pairs $Z\subset H$ in a Calabi-Yau threefold $X$. Here $H$ is a member of a sufficiently positive linear system and $Z$ is a 1-dimensional subscheme of it. The associated sheaf is the ideal sheaf of $Z\subset H$, pushed forward to $X$ and considered as a certain Joyce-Song pair in the derived category of $X$. We express these invariants in terms of the MNOP invariants of $X$.

\end{abstract}
\section{Introduction} \label{sec:intro}
\hfill Throughout this paper we fix a Calabi-Yau 3-fold $X$ with $H^1(\O_X)=0$.
\\ String theorists conjecture the modularity of certain generating functions of counts of supersymmetric D4-D2-D0 BPS states of $X$ \cite{a78,a95,a100}. Motivated by this we define invariants of $X$ which count 0- and 1-dimensional subschemes (i.e. D2-D0 branes) of hyperplanes (D4 branes) in $X$. We briefly discuss the relationship to the physicists' S-duality conjecture in Section \ref{ami}. 

We wish to count pairs of the form
\begin{equation} \label{triple}
Z\subset H\quad\mathrm{in\ \ }X,
\end{equation}
where $H$ is a suitably ample hypersurface and $Z$ is a 1-dimensional subscheme of $H$. While we demand that $H$ be pure of dimension 2 (i.e. a divisor in $X$), the subscheme $Z$ need not be pure, so is in general the union of a curve and a 0-dimensional subscheme.

Since $H^1(X,\O_X)=0$, all deformations of $H$ are in the same linear system $$|H|:=\P(H^0(\O_X(H))).$$ This carries a universal hyperplane
\begin{equation} \label{uH}
\xymatrix@C=-5pt@R=15pt{
\H \ar[d] & \subset\,X\times|H| \\ |H|.}
\end{equation}
The data \eqref{triple} are naturally parameterized by the relative Hilbert scheme of 1-dimensional subschemes of the fibers of the family \eqref{uH}.
That is, if we fix $\beta\in H_2(X,\Z)$ and $n\in\Z$, then the moduli space of pairs \eqref{triple} is the relative Hilbert scheme
\begin{equation} \label{hilber}
\hilb_{\beta,n}(\H/|H|)
\end{equation}
of 1-dimensional subschemes $Z$ of the universal hyperplane whose pushforward to $X$ has fundamental class $\beta=[Z]=\ch_2(i_*\O_Z)$ and holomorphic Euler characteristic $n=\chi(\O_Z)$.

We will produce a symmetric perfect obstruction theory on \eqref{hilber} by considering it as a moduli space of torsion sheaves of the form
\begin{equation} \label{shea}
i_*I_Z\,\in\,\operatorname{Coh}(X).
\end{equation}
Here $i\colon H\hookrightarrow X$ is the inclusion, and $I_Z$ denotes the ideal sheaf of $Z$ considered as a subscheme of $H$. So \eqref{shea} is a torsion sheaf on $X$ with rank 1 on its 2-dimensional support $H$. When $Z\subset H$ is a Cartier divisor, it is the pushforward of a line bundle on $H$.

\begin{theorem} \label{thm:torsion}
Suppose that $H$ is sufficiently positive with respect to $\beta\in H_2(X;\Z)$ and $n\in \Z$, in the sense of Definition \ref{def:pos}. Then 
$$\hilb_{\beta,n}(\mathcal H/|H|)$$ is a locally complete moduli space of torsion sheaves \eqref{shea} on $X$. It admits a symmetric perfect obstruction theory and so a virtual cycle of virtual dimension 0.
\end{theorem}

There are two obvious problems to overcome in proving Theorem \ref{thm:torsion}:
% forming a moduli space of sheaves such as \eqref{shea}:

\begin{itemize}
\item The sheaves \eqref{shea} need not be (Gieseker) semistable when $H$ is reducible or nonreduced.
\item Deformations of \eqref{shea} might not, \emph{a priori}, be sheaves of the same form. They could be arbitrary torsion sheaves of the same topological type, like the pushforward $i_*L$ of a general line bundle on a hyperplane $H$ (rather than one which is a subsheaf of $\O_H$).
\end{itemize}

We circumvent these by using certain Joyce-Song pairs \cite{a30}. These come with a different notion of stability which gets round the first problem. And for $H\gg0$, they allow us to see that deformations of \eqref{shea} are also push forwards of ideal sheaves, so that $\hilb_{\beta,n}(\H/|H|)$ is indeed an open and closed subscheme of the stack of all coherent sheaves on $X$.

Our Joyce-Song pairs are of the form
$$
\Big(\I_{Z\,},\,s\in H^0(\I_Z(H))\Big),
$$
where $Z$ and $H$ are as before, but $\I_Z$ denotes the ideal sheaf of $Z$ when considered as a subscheme of $X$ (rather than $H$).\footnote{We use straight $I$s to denote ideal sheaves on $H$ and curly $\I$s for ideal sheaves on $X$.} Since the torsion-free rank-1 sheaf $\I_Z$ is automatically Gieseker stable, the only further stability condition we need is that $s$ should be nonzero. Therefore we get an injection
\begin{equation} \label{JSideal}
\O(-H)\stackrel s\longrightarrow\I_Z
\end{equation}
whose cokernel $i_*I_Z$ is the torsion sheaf \eqref{shea}. That is, if we let $I^\bullet\in D(X)$ denote the two term complex formed by \eqref{JSideal}, then
$$
\text{\emph{the Joyce-Song complexes $I^\bullet$ \eqref{JSideal} are quasi-isomorphic to the torsion sheaves \eqref{shea}.}}
$$
More globally we show in Proposition \ref{tube} that this observation gives an isomorphism between
\begin{itemize}
\item an open and closed subscheme of the moduli space of Joyce-Song pairs \eqref{JSideal}, and
\item the relative Hilbert scheme \eqref{hilber} parameterizing the sheaves \eqref{shea}.
\end{itemize}
Finally, for $H\gg0$, Joyce and Song show that the deformation theory of the complexes $I^\bullet\in D(X)$ gives this space a symmetric perfect obstruction theory. Since $I^\bullet$ is quasi-isomorphic to $i_*I_Z\in D(X)$, this shows that the deformation theory of the sheaf $i_*I_Z$ indeed endows $\hilb_{\beta,n}(\H/|H|)$ with a symmetric perfect obstruction theory.

\medskip

Using Theorem \ref{thm:torsion} we can define an invariant
$$
N^H_{\beta,n}(X):=\int_{[\hilb_{\beta,n}(\mathcal H/|H|)]^{vir}}1.
$$
Since the obstruction theory is symmetric, by \cite{a1} this can also be written as a weighted Euler characteristic of $\hilb_{\beta,n}(\mathcal H/|H|)$.  The invariants $N^H_{\beta,n}$ are closely related to the MNOP invariants $I_{\beta,n}(X)$ counting the ideal sheaves $\I_Z$ of $X$.

\begin{theorem} \label{thm:mnop} Under the conditions of Theorem \ref{thm:torsion} we have
$$N^H_{\beta,n}=(-1)^{\mathsf{c}-1}\,\mathsf{c} \cdot I_{\beta,n}\,,$$ where
$\mathsf{c}=\mathsf{c}(\beta,n,H,X)$ is the topological number
$$
\int_X\Big(\frac16H^3+H^2\operatorname{td}_2(X)
-H \cdot \beta\Big)-n.
$$
\end{theorem}

\begin{example} \label{gaiotto}
Most of the examples worked out explicitly in \cite[Sections 2.1 to 2.6]{a95} fit into the setting of this section and hence, the corresponding invariants are captured by the invariants $N^H_{\beta,n}$.
As an illustration, we work out the following two simple cases for linear hyperplane sections $H$ of the quintic Calabi-Yau threefold $X$.
\begin{itemize}
\item $\I_p$ is the ideal sheaf of a point $p\in X$, so $\beta=0$, $n=1$ and $$\mathsf{c}=\operatorname{td}_2(X)\cdot H+H^3/6-1=25/6+5/6-1=4.$$ Also $I_{0,1}=-\chi(X)=200$, so Theorem \ref{thm:mnop} gives
$$N^{H}_{0,1}=(-1)^3\times4\times200=-800.$$
\item $\I_C$ is the ideal sheaf of a line $C\subset X$, so $\beta \cdot H=1,\ n=1$, and hence $$\mathsf{c}=25/6+5/6-1=3.$$ Also $I_{\beta,1}=2875$, the number of rational degree 1 curves in a generic quintic, giving
$N^H_{\beta,1}=(-1)^2\times3\times2875=8625$. \hfill $\oslash$
\end{itemize}
\end{example}

\subsection*{Acknowledgments}
We thank Kai Behrend, Vincent Bouchard, Jim Bryan, Tudor Dimofte, Dagan Karp, Jan Manschot, Davesh Maulik and Yukinobu Toda for useful discussions. The second author would like to thank the University of British Columbia, the Max Planck Institute, Bonn and the Isaac Newton Institute for Mathematical Sciences, Cambridge for their hospitality. The third author is supported by the EPSRC programme grant EP/G06170X/1.

\section{Proof of Theorem \ref{thm:torsion}}\label{sec:MNOP}

A Joyce-Song stable pair \cite{a30} is a pair
$$
\Big(E,\,s\in H^0(E(H))\Big),
$$
such that
\begin{itemize}
\item $E$ is coherent sheaf on $X$, Gieseker semistable with respect to $H$, and
\item $s$ is a section which does not factor through any destabilizing subsheaf of $E$. 
\end{itemize}

A family of such pairs over a base scheme $B$ is a sheaf $\E$ over $X\times B$, flat over $B$, and a section of $\E\otimes\pi_X^*\O(H)$ such that the restriction of $(\E,s)$ to any fiber $X\times\{b\}$ is a stable pair in the above sense. Fixing a Chern character $c\in H^{\mathrm{ev}}(X,\Q)$, there is a projective scheme $$J_c(X)$$ representing the moduli functor which assigns to a scheme $B$ the set of isomorphism classes\footnote{Isomorphism is meant in the strict sense: two families of stable pairs $(\E_i,s_i)$ are isomorphic if there is an isomorphism $\E_1\to\E_2$ taking $s_1$ to $s_2$. Stable pairs have no automorphisms so there is no need to tensor by line bundles pulled back from the base (as one does to define isomorphism for families of stable sheaves, for example: stable sheaves have Aut\,$=\CC^*$).} of Joyce-Song stable pairs over $B$.

Forgetting the section $s$ and passing to the S-equivalence class of the sheaf $E$ gives a morphism
$$
J_{c}(X)\longrightarrow M_{c}(X)
$$
to the moduli space of Gieseker semistable sheaves $E$ of Chern character $c$.

We take $$c=(1,0,-\beta,-n).$$ Then the Hilbert scheme $ I_n(X,\beta)$ of ideal sheaves $\I_Z$ defines an open and closed subscheme\footnote{In fact it is all of $M_{c}(X)$ when $H^2(X,\Z)$ is torsion-free.} of $M_{c}(X)$ (see for instance the proof of Theorem 2.7 in \cite{a17} for a careful proof of this fact). Thus its inverse image
$$
J_n(X,\beta):=J_{c}(X)\times_{M_{c}(X)} I_n(X,\beta)
$$
is both open and closed in $J_{c}(X)$. It therefore has the same deformation theory as $J_{c}(X)$, so we can use the results of Joyce and Song.

\begin{prop} \label{tube}
$J_n(X,\beta)\cong\hilb_{\beta,n}(\H/|H|)$.
\end{prop}

\begin{proof}
$\hilb_{\beta,n}(\H/|H|)$ carries a pair of universal subschemes
\begin{equation} \label{notn}
\mathcal Z\subset\H\subset X\times\hilb_{\beta,n}(\H/|H|).
\end{equation}
Letting $i$ denote the second inclusion, we get the right hand column and central row of the following commutative diagram of short exact sequences.
\begin{equation} \label{comdg}
\xymatrix@=16pt{
&&0 \ar[d] & 0 \ar[d] \\
0 \ar[r] & \O_{X\times \hilb}(-\H)\ar[r]^(.65)s \ar@{=}[d] & \I_{\mathcal Z} \ar[r]\ar[d]             & i_*I_{\mathcal Z} \ar[d]\ar[r] & 0 \\
0 \ar[r] & \O_{X\times \hilb}(-\H)\ar[r] & \O_{X\times \hilb} \ar[r] \ar[d]  & i_*\O_{\H} \ar[d]\ar[r] & 0 \\
&&  i_*\O_{\mathcal Z} \ar@{=}[r]\ar[d] & i_*\O_{\mathcal Z}\ar[d] \\
&&0 & 0.\!}
\end{equation}
Filling in the diagram gives the top row, producing the flat family of stable pairs
\begin{equation} \label{amici}
\O\stackrel s\longrightarrow\I_{\mathcal Z}(\H-\pi_X^*H)\otimes\pi_X^*\O(H)
\end{equation}
over the base $\hilb_{\beta,n}(\H/|H|)$. This is classified by a map from the base to the moduli space of stable pairs:
\begin{equation} \label{oneway}
\hilb_{\beta,n}(\H/|H|)\longrightarrow J_n(X,\beta).
\end{equation}

\medskip
Similarly, $X\times J_n(X,\beta)$ carries a universal stable pair
\begin{equation} \label{sect}
\O\longrightarrow\mathcal E\otimes\pi_X^*\O(H),
\end{equation}
flat over $J_n(X,\beta)$. Since the restriction of $\mathcal E$ to each $X$-fiber is torsion-free of rank 1, its double dual $\mathcal E^{\vee\vee}$ is locally free by \cite[Lemma 6.13]{a143}. Therefore it defines a map from $J_n(X,\beta)$ to $\operatorname{Pic}(X)$ which takes closed points to the trivial line bundle $\O_X$. But $H^1(X,\O_X)=0$ so $\operatorname{Pic}(X)$ is a union of discrete reduced points and the map is constant. Pulling back a Poincar\'e line bundle shows that $\mathcal E^{\vee\vee}$ is the pullback $\pi_J^*L$ of some line bundle $L$ on $J_n(X,\beta)$. Therefore $\mathcal E\subset \mathcal E^{\vee\vee}$ must take the form
\begin{equation} \label{Eideal}
\mathcal E\,\cong\,\pi_J^*L\otimes\I_{\mathcal Z}
\end{equation}
for some subscheme $\mathcal Z\subset X\times J_n(X,\beta)$. Since $\mathcal E$ is flat over $J_n(X,\beta)$, so $\mathcal Z$ must be too.

Composing the section \eqref{sect} with $\mathcal E\to\mathcal E^{\vee\vee}$ gives a section
\begin{equation} \label{ssec}
\O_{X\times J_n(X,\beta)}\longrightarrow\O(H)\boxtimes L,
\end{equation}
nonzero on each $X$-fiber by stability. Its zero locus is a divisor $$\H\in|\O(H)\boxtimes L|,$$ giving a classifying map $J_n(X,\beta)\to |H|$ such that the pullback of $\O_{X\times|H|}(\H)$ is $\O(H)\boxtimes L$. Since the section \eqref{ssec} factors through \eqref{Eideal} it follows that
$$
\mathcal Z\subset\mathcal H,
$$
giving us a classifying map from our base $J_n(X,\beta)$ to $\hilb_{\beta,n}(\H/|H|)$.
By inspection it is the inverse of the map \eqref{oneway}.
\end{proof}

\begin{defn}\label{def:pos} Assume that $\beta\in H_2(X;\Z)$ and  $n\in \Z$ are given. We say $H$ is sufficiently positive with respect to $\beta,n$ if 
\begin{equation}\label{van}
H^i(X,\I_Z(H))=0
\end{equation} 
for $i>0$ for any ideal sheaf $\I_Z\in I_n(X,\beta)$. For fixed $(\beta,n)$ the ideal sheaves $\I_Z\in I_n(X,\beta)$ form a bounded family, so \eqref{van} is satisfied for all $H$ sufficiently positive.
\hfill $\oslash$
\end{defn} 

\begin{cor} \label{deff} Suppose that $H^1(X,\O_X)=0$ and $\beta,n,H$ satisfy \eqref{van}.
Then, using the same notation as in \eqref{notn}, the deformation-obstruction theory of the sheaves $i_*I_Z$ \eqref{shea} given by \cite[Section 4.4]{a10},
\begin{equation} \label{arrow}
\tau^{[1,2]}R\pi_{H*}R\Home(i_*I_{\mathcal Z},i_*I_{\mathcal Z})[2]\longrightarrow
\mathbb L_{\hilb_{\beta,n}(\H/|H|),}
\end{equation}
defines a symmetric perfect obstruction theory on $\hilb_{\beta,n}(\H/|H|)$.
\end{cor}

In \eqref{arrow}, $\pi_H$ denotes the projection from $X\times\hilb_{\beta,n}(\H/|H|)$ to its second factor. We recall the definition of the arrow \cite[Section 4]{a10}. Project the Atiyah class
$$
\operatorname{At}(i_*I_{\mathcal Z})\,\in\,\operatorname{Ext}^1\big(i_*I_{\mathcal Z},\,i_*I_{\mathcal Z}\otimes\mathbb L_{X\times\hilb}\big)
$$
to
$$
\operatorname{Ext}^1\big(i_*I_{\mathcal Z},\,i_*I_{\mathcal Z}\otimes
\pi_H^*\mathbb L_{\hilb}\big)
$$
and consider the resulting element as a map
\begin{equation} \label{At0}
\pi_H^*\mathbb L_{\hilb}^\vee\longrightarrow
R\Home(i_*I_{\mathcal Z},i_*I_{\mathcal Z})[1].
\end{equation}
By adjunction this induces
\begin{equation} \label{At}
\mathbb L_{\hilb}^\vee\longrightarrow
R\pi_{H*}R\Home(i_*I_{\mathcal Z},i_*I_{\mathcal Z})[1].
\end{equation}
Next we project to the $\tau^{\ge0}$ truncation of the right hand side. In turn this receives a map from the truncation $\tau^{[0,1]}$, and in \cite[Section 4.4]{a10} it is shown that the map from $\mathbb L_{\hilb}^\vee$ lifts uniquely to it. Dualizing and using Serre duality down the fibers of $\pi_H$ gives \eqref{arrow}.

\begin{proof}[Proof of Corollary \ref{deff}.]
For $H\gg0$, Joyce and Song \cite[Theorem 12.20]{a30} give a symmetric perfect obstruction theory on $J_c(X)$:
\begin{equation} \label{arro}
\tau^{[1,2]}R\pi_{J*}R\Home(\bI^\bullet,\bI^\bullet)[2]\longrightarrow\mathbb L_{J_c(X)}.
\end{equation}
Here $\pi_J$ is the projection $X\times J_c(X)\to J_c(X)$ and $\bI^\bullet\in D(X\times J_c(X))$ is the 2-term complex
$$
\O_{X\times J_c(X)}(-\pi_X^*H)\longrightarrow\E
$$
defined by the universal stable pair $\O\to\E\otimes\pi_X^*\O(H)$. (From now on we suppress the $\pi_X^*$.) The arrow in \eqref{arro} is the composition of Serre duality for $R\pi_{J*}R\Home(\bI^\bullet,
\bI^\bullet)$ and the Atiyah class of $\bI^\bullet$, just as in \eqref{At}.

We restrict their result to $\hilb_{\beta,n}(\H/|H|)$ using Proposition \ref{tube}.
Their precise $H\gg0$ condition is that the sheaves in their stable pairs are $H$-regular; this becomes the cohomology vanishing condition \eqref{van}. The first row of the diagram \eqref{comdg} gives the quasi-isomorphism
$$
\bI^\bullet\otimes\O(H-\H)\ =\ \big\{\O_{X\times\hilb_{\beta,n}(\H/|H|)}(-\H)\to\I_{\mathcal Z}\big\}\ \simeq\ i_*I_{\mathcal Z}.
$$
Using the line bundle $\O(1)$ pulled back from the projective space $|H|$, we have the isomorphism
$$
\O(\mathcal H)\cong\O(H)\boxtimes\O(1).
$$
In particular, $\bI^\bullet\simeq i_*I_{\mathcal Z}(\H-H)\cong i_*I_{\mathcal Z}(1)$, so \eqref{arro} can be written
$$
\tau^{[1,2]}R\pi_{H*}R\Home(i_*I_{\mathcal Z}(1),i_*I_{\mathcal Z}(1))[2]
\longrightarrow\mathbb L_{\hilb_{\beta,n}(\H/|H|).}
$$
Again the arrow is given by the composition of Serre duality and the Atiyah class of $i_*I_{\mathcal Z}(1)$.

This is almost identical to (\ref{arrow}, \ref{At}) except for the twist by $\O(1)$. Now
$$
\operatorname{At}(i_*I_{\mathcal Z}(1))\,=\,\operatorname{At}(i_*I_{\mathcal Z})+\operatorname{id}_{i_*I_{\mathcal Z}}\otimes\operatorname{At}(\O(1))
$$
in $\operatorname{Ext}^1\big(i_*I_{\mathcal Z}(1),\,i_*I_{\mathcal Z}(1)\otimes
\pi_H^*\mathbb L_{\hilb}\big)\cong\operatorname{Ext}^1\big(i_*I_{\mathcal Z},\,i_*I_{\mathcal Z}\otimes\pi_H^*\mathbb L_{\hilb}\big)$. That is, tensoring with $\O(1)$ changes the map \eqref{At0} by addition of the following composition
$$\xymatrix@=40pt{
\pi_H^*\mathbb L_{\hilb}^\vee \ar[r]^(.46){At(\O(1))} &
\O_{X\times\hilb}[1] \ar[r]^(.38){\operatorname{id}_{i_*I_{\mathcal Z}}} & R\Home(i_*I_{\mathcal Z},i_*I_{\mathcal Z})[1].}
$$
Therefore \eqref{At} changes by addition of the composition
$$\xymatrix@=40pt{
\mathbb L_{\hilb}^\vee \ar[r]^(.36){At(\O(1))} &
R\pi_{H*}\O_{X\times\hilb}[1] \ar[r]^(.42){\operatorname{id}_{i_*I_{\mathcal Z}}} & R\pi_{H*}R\Home(i_*I_{\mathcal Z},i_*I_{\mathcal Z})[1].}
$$
The truncation procedure gives unique lifts to the $\tau^{[0,1]}$ truncation of the central and right hand terms. But our $H^1(X,\O_X)=0$ condition means that $\tau^{[0,1]}$ applied to the central term is zero.
\end{proof}

Corollary \ref{deff} completes the proof of Theorem \ref{thm:torsion}.

\section{Proof of Theorem \ref{thm:mnop}}
By Corollary \ref{deff} we define the invariants associated to $\hilb_{\beta,n}(\H/|H|)$. These invariants give a virtual count of the pairs \eqref{triple}. Then we will relate them to the MNOP invariants $I_{\beta,n}$ counting subschemes $Z\subset X$.

\begin{defn} \label{defn} Let $X$ be a Calabi-Yau 3-fold with $H^1(X,\O_X)=0$ and suppose that $H$ is sufficiently positive with respect to $\beta,n$ in the sense of Definition \ref{def:pos}.

The perfect obstruction theories of Corollary \ref{deff} and \cite{a20} respectively endow $\hilb_{\beta,n}(\H/|H|)$ and $ I_n(X,\beta)$ with virtual cycles of dimension zero. We define
$$
N^H_{\beta,n}:=\int_{[\hilb_{\beta,n}(\H/|H|)]^{vir}}1\,,\qquad
I_{\beta,n}:=\int_{[ I_n(X,\beta)]^{vir}}1\,.
$$
Since these obstruction theories are symmetric, by \cite{a1} the invariants can also be written as weighted Euler characteristics
\begin{eqnarray} \label{Ech}
N^H_{\beta,n} &=& \chi\big(\!\hilb_{\beta,n}(\H/|H|),\nu_{\hilb_{\beta,n}(\H/|H|)}\big),
\\ I_{\beta,n} &=& \chi\big( I_n(X,\beta),\nu_{ I_n(X,\beta)}\big), \nonumber
\end{eqnarray}
where $\nu_M$ denotes the Behrend function \cite{a1} of a scheme $M$. 
\hfill $\oslash$
\end{defn}

By the usual arguments \cite{a10, a30, a20} these invariants are unchanged under smooth deformation of $X$.

There is an obvious forgetful map
\begin{equation} \label{forget}
\hilb_{\beta,n}(\H/|H|)\longrightarrow I_n(X,\beta), \qquad (Z\subset H)\mapsto Z.
\end{equation}
The fiber over $Z\subset X$ is $\P(H^0(\I_Z(H)))$. More globally we have the following result. We use the notation
$$
\mathcal Z\subset X\times I_n(X,\beta)\stackrel{\pi_I}\longrightarrow I_n(X,\beta)
$$
for the universal subscheme and the projection to the second factor.

\begin{lemma} \label{lemming} Supposing again that $\beta,n,H$ satisfy \eqref{van}, then the map \eqref{forget} is the projective bundle
$$\hilb_{\beta,n}(\H/|H|)\cong\P(\pi_{I*}\I_{\mathcal Z}(H)).$$
\end{lemma}

\begin{proof} The cohomology vanishing condition \eqref{van} ensures that $\pi_{I*}\I_{\mathcal Z}(H)$ is a vector bundle. Its pullback to $\P(\pi_{I*}\I_{\mathcal Z}(H))$ carries a tautological subbundle
$$
\O(-1)\hookrightarrow\pi_{I*}\I_{\mathcal Z}(H),
$$
where again we have suppressed the pullback map. By adjunction we get
$$
\pi_I^*\O(-1)\hookrightarrow\I_{\mathcal Z}(H)
$$
and so a nowhere vanishing section of $\I_{\mathcal Z}(H)\boxtimes\O(1)=\I_{\mathcal Z}(\mathcal H)$. This is a family of stable pairs over the base $\P(\pi_{I*}\I_{\mathcal Z}(H))$, classified by a map
\begin{equation} \label{pam}
\P(\pi_{I*}\I_{\mathcal Z}(H))\longrightarrow J_n(X,\beta)\,\cong\,\hilb_{\beta,n}(\H/|H|).
\end{equation}

\medskip Conversely, $X\times\hilb_{\beta,n}(\H/|H|)$ carries the universal family of stable pairs \eqref{amici}. Twisting by $\O(-1)$ (pulled back from $|H|$) gives
$$
\O(-1)\stackrel s\longrightarrow\I_{\mathcal Z}(H).
$$
Pushing down to $\hilb_{\beta,n}(\H/|H|)$ gives
\begin{equation} \label{lsb}
\O(-1)\hookrightarrow\pi_{H*}\I_{\mathcal Z}(H).
\end{equation}
Now $\pi_{H*}\I_{\mathcal Z}(H)$ on $\hilb_{\beta,n}(\H/|H|)$ is the pullback of the bundle $\pi_{I*}\I_{\mathcal Z}(H)$ on $I_n(X,\beta)$ via the map \eqref{forget}. Therefore the line subbundle \eqref{lsb} is
classified by a map $\hilb_{\beta,n}(\H/|H|)\to\P(\pi_{I*}\I_{\mathcal Z}(H))$, the inverse of \eqref{pam}.
\end{proof}

\begin{cor} \label{corblimey} For $H^1(X,\O_X)=0$ and $\beta,n,H$ satisfying \eqref{van}, the invariants of Definition \ref{defn} satisfy
$$
N^H_{\beta,n}=(-1)^{\mathsf{c}-1}\,\mathsf{c}\cdot I_{\beta,n\,},
$$
where $\mathsf{c}=\mathsf{c}(\beta,n,H,X)$ is the topological number
$$
\mathsf{c}=\chi(\I_Z(H))=\int_X\Big(\frac16H^3+H^2\operatorname{td}_2(X)
-H \cdot \beta\Big)-n.
$$
\end{cor}

\begin{proof} By Lemma \ref{lemming}, the map \eqref{forget}
$$
\hilb_{\beta,n}(\H/|H|)\longrightarrow I_n(X,\beta)
$$
is smooth of relative dimension $\chi(\I_Z(H))-1$. Therefore by \cite[Proposition 1.5(i)]{a1}, the Behrend function of $\hilb_{\beta,n}(\H/|H|)$ is the pullback of that of $ I_n(X,\beta)$, multiplied by $(-1)^{\chi(\I(H))-1}$. Since the fibers of the map \eqref{forget} have Euler characteristic $\chi(\I_Z(H))$, the formulas \eqref{Ech} give the result.
\end{proof}

Corollary \ref{corblimey} finishes the proof of Theorem \ref{thm:mnop}. Note that if $H^2(X,\Z)$ is torsion-free then $I_n(X,\beta)\cong M_{c}(X)$ for $c=(1,0,-\beta,-n)$, and $J_{c}(X)\cong \hilb_{\beta,n}(\H/|H|)$ by Proposition \ref{tube}. In this case Corollary \ref{corblimey} is a special case of the wall-crossing formula \cite[Theorem 5.27]{a30}.

\section{S-duality and modularity} \label{ami}
Our motivation for defining the invariants $N^H_{\beta,n}$ was to try to understand how to define the ``\emph{supersymmetric BPS invariants associated to supersymmetric D4-D2-D0 systems}" studied by string theorists \cite{a78,a95,a100}. Based on relation of D4-D2-D0 branes with 2-dimensional conformal field theory \cite{MSW97}, the generating series of these BPS invariants are argued to be modular forms \cite{CDDMV06, a78}. This remarkable modularity can also be deduced from the S-duality (electric-magnetic duality) of string theory \cite{DM07}.

Most of the examples of D4-D2-D0 systems studied in \cite[Sections 2.1--2.6]{a95} are of the form \eqref{shea} above, which is what led to our definition. However it seems that in general one should count \emph{all} (semi)stable torsion sheaves with 2-dimensional support of the right topological type, not just those of the form \eqref{shea}. Assuming that the physical count is given by Joyce-Song's generalized DT invariant, the S-duality predictions take the following form.

Assume for simplicity that $\operatorname{Pic} X$ is generated by an ample divisor $L$. Fix $k\in \Z_{>0}$ and let $H$ be a smooth element of $|kL|$. Denoting the generator of $H^4(X,\Z)\cong \Z$ by $\ell$ and fixing $i,n\in \Z$, we
let $$c=c(i,n)=(0,\,H,\,H^2/2-i \ell,\,\chi(\O_{H})-H\cdot \operatorname{td}_2 (X)-n).$$ Then there is a moduli space $M_c(X)$ of Gieseker-semistable sheaves on $X$ with Chern character $c$, and Joyce-Song's generalized Donaldson-Thomas invariant $\overline{DT}(c)\in\Q$ \cite{a30}. We define the generating function of these invariants with fixed 1st and 2nd Chern characters by $$Z^{H}_i(q)=\sum_{n\in \Z}\overline{DT}(c(i,n))q^n.$$
Tensoring by $\O(\pm L)$ induces isomorphisms of the moduli spaces $M_c(X)$, so the generating series $Z^{H}_i(q)$ and $Z^{H}_{i\pm kL^3}(q)$ only differ by a shift in the power of $q$. 

According to \cite{a78, a95} the vector of generating series $$\Big(q^{a_i}Z^{H}_i(q)\Big)_{i=0}^{kL^3-1}$$ should be a holomorphic vector-valued modular form of weight $-3/2$, where
$$a_i\ =\ \frac{(2i+H^3)^2}{8H^3}-\frac{H^3}{8}-\frac{\chi(H)}{24}\ \in\ \Q.$$
When all of the sheaves \eqref{shea} are stable\footnote{For instance this is the case when all members of the linear system $|H|$ are reduced and irreducible. The hyperplane sections of the quintic threefold have this property by the Lefschetz hyperplane theorem.} one can think of $N^H_{\beta,n}$ as the contribution of the component $\hilb_{\beta,n}(\H/|H|)$ of the moduli space of torsion sheaves to the physicists' numbers. But more generally one would expect to be able to relate our invariants $N^H_{\beta,n}$ to the $\overline{DT}(c)$ via a sequence of wall crossings.\footnote{In \cite{a114} Toda studies some different but remarkable wall crossings. These relate the counting of torsion sheaves to the counting of both ideal sheaves \cite{a15} and stable pairs \cite{a17}. There may be some connection: a hint is provided by the fact that Fujita's conjecture (which would determine which $H$ are sufficiently positive in the sense of Definition \ref{def:pos}) enters into his analysis. Manschot \cite{M10} also studies another wall crossing involving these invariants.} Ideally these would be in the space of Bridgeland stability conditions on $D(X)$, starting from a stability condition that approximates Joyce-Song stability for the complexes \eqref{JSideal}, and ending with one approximating Gieseker stability for the quasi-isomorphic sheaves \eqref{shea}. 

In combination with Corollary \ref{corblimey} this would express the MNOP invariants $I_{\beta,n}$ in terms of modular forms via universal formulae arising from the wall crossing. This would constrain the $I_{\beta,n}$ enormously, so that calculation of a finite number of these invariants should determine all the rest. We plan to return to this in future work.

\noindent {\tt{amingh@math.umd.edu}} \\
\noindent {\small University of Maryland \\ College Park, MD 20742-4015, USA} \\

\noindent {\tt{sheshmani.1@math.osu.edu}} \\
\noindent{\small Department of mathematics at the Ohio State University, 600 Math Tower, 231 West 18th Avenue, Columbus, OH 43210-1174} \\

\noindent {\tt{rpwt@ic.ac.uk}} \\
\noindent{\small Imperial College London \\ 180 Queen's Gate, London SW7 2AZ, UK}

\end{document}